\documentclass[a4paper,10pt, preprint, 3p, times]{P-article}
\usepackage{amssymb}
\usepackage{epsfig}
\usepackage{amsfonts}
\usepackage{amsmath}
\usepackage{euscript}
\usepackage{amscd}
\usepackage{amsthm}
\DeclareMathAlphabet{\mathpzc}{OT1}{pzc}{m}{it}
\newtheorem{thm}{Theorem}[section]
\newtheorem{lem}[thm]{Lemma}
\newtheorem{prop}[thm]{Proposition}

\newtheorem{cor}[thm]{Corollary}

\newdefinition{defn}[thm]{Definition}
\newdefinition{ex}[thm]{Example}
\newdefinition{rem}[thm]{Remark}
\newdefinition{note}{Note}
\newdefinition{q}{Question}

\newcommand{\comment}[1]{}

\newcommand\m {\mathfrak{m}}

\newcommand{\A}[1]{\mathbb{A}^{#1}}
\newcommand{\ol}[1] {\overline{#1}}
\newcommand{\ul}[1] {{\bf\underline{#1}}}

\newcommand{\nhkrghtarrw}{\hookrightarrow \hspace*{-.50cm}{/}\hspace*{.2cm}}

\begin{document}
\begin{frontmatter}
\title{A Note on Residual Variables of an Affine Fibration}

\author{Prosenjit Das}
\address{Department of Mathematics, Indian Institute of Space Science and Technology, \\
Valiamala P.O., Trivandrum 695 547, India\\
email: \texttt{prosenjit.das@gmail.com}}
\author{Amartya K. Dutta}
\address{Stat-Math Unit, Indian Statistical Institute,\\
203, B.T. Road, Kolkata 700 108, India\\
email: \texttt{amartya@isical.ac.in}}

\begin{abstract}
    In a recent paper \cite{Kahoui_Residual}, M.E. Kahoui has shown that if $R$ is a polynomial ring over $\mathbb C$, $A$ an $\A{3}$-fibration over $R$, and $W$ a residual variable of $A$ then $A$ is stably polynomial over $R[W]$. In this article we show that the above result holds over {\it any} Noetherian domain $R$ provided the module of differentials $\Omega_R(A)$ of the affine fibration $A$ (which is necessarily a projective $A$-module by a theorem of Asanuma) is a stably free $A$-module. 

\noindent
{\scriptsize Keywords: Residual variable; Stably polynomial algebra; $\A{n}$-fibrations, Module of differentials.}\\
{\scriptsize {\bf Mathematics Subject classifications (2010)}: Primary 14R25; Secondary 13B25, 13F20}
\end{abstract}
\end{frontmatter}

\section{Introduction} \label{sec_intro}
Let $R$ denote a Noetherian integral domain. A finitely generated flat $R$-algebra $A$ is said to be an $\A{n}$-fibration over $R$ if $A \otimes_R k(P)$ is a polynomial ring in $n$ variables over $k(P)$ for every prime ideal $P$ of $R$. The concept of residual variable was defined by S.M. Bhatwadekar and A.K. Dutta in \cite{BD_RES} as an element $W$ in the polynomial algebra $R[X,Y]$ for which $R[X,Y] \otimes_R k(P)$ is a polynomial algebra in one variable over $R[W] \otimes_R k(P)$ for every prime ideal $P$ of $R$. The following statement is a part of their main result (\cite{BD_RES}, Theorem 3.1) on residual variables: 

\begin{thm} \label{Thm_BD_Res-1}
Let $R$ be a Noetherain domain and W an element of $R[X,Y]$. Then the following are equivalent:
\begin{enumerate}
\item [\rm (1)]$W$ is a residual variable in $R[X,Y]$.
\item [\rm (2)]$R[X,Y]$ is a stably polynomial algebra over $R[W]$, i.e., $R[X,Y][Z_1, Z_2, \cdots, Z_r]=R[W][T_0, T_1, \cdots, T_r]$ for some indeterminates $Z_i$'s, $T_i$'s over $R$.
\end{enumerate}

\end{thm}

They also observed (\cite{BD_RES}, Remark 3.4) that an analogous result holds for a system of $m$ algebraically independent elements $W_1,W_2, \cdots, W_m$ in the polynomial ring $R[X_1, X_2, \cdots, X_{m+1}]$. Recently El Kahoui \cite{Kahoui_Residual} has extended the concept of residual variable of a polynomial ring to that of an $\A{n}$-fibration: he defines an element $W$ of an $\A{n}$-fibration $A$ over a ring $R$ to be a residual variable of $A$ over $R$ if $A \otimes_R k(P)$ is a polynomial algebra in $n-1$ variables over $R[W] \otimes_R k(P)$ for each prime ideal $P$ of $R$. He shows that when $n=3$ and the base ring $R$ is a polynomial algebra over $\mathbb C$, then the extended concept satisfies the following analogue of Theorem \ref{Thm_BD_Res-1} (\cite{Kahoui_Residual}, Theorem 3.4):

\begin{thm} \label{Thm_Kahoui_Res-1}
Let $R$ be a finite-dimensional polynomial algebra over $\mathbb C$, $A$ an $\A{3}$-fibration over $R$ and $W$ an element of $A$. Then the following are equivalent:
\begin{enumerate}
\item [\rm (1)]$W$ is a residual variable of $A$ over $R$.
\item [\rm (2)]$A[Z_1, Z_2, \cdots, Z_r] = R[W][T_0, T_1, \cdots, T_r, T_{r+1}]$ for some for some indeterminates $Z_i$'s, $T_i$'s over $R$.
\end{enumerate}

\end{thm}

Kahoui also observes in \cite{Kahoui_Residual} that if the hypothesis ``$R$ is a polynomial algebra over $\mathbb C$'' is weakened to ``$R$ is any domain containing $\mathbb Q$'', then the conclusion (2) implies (1)  still holds in Theorem \ref{Thm_Kahoui_Res-1}. He remarks (\cite{Kahoui_Residual}, Pg. 39) that it is not known whether the converse also holds. Example \ref{example-1} of our paper shows that the converse does not hold in general
even when R is a regular factorial affine domain. However, using the techniques of Bhatwadekar-Dutta, we shall show that the converse (i.e., $(1) \implies (2)$ of Theorem \ref{Thm_Kahoui_Res-1}) indeed holds over any Noetherian domain when the module of differentials $\Omega_R(A)$ is stably free over $A$. More generally, we shall prove (Corollary \ref{our_res_prop}, Theorem \ref{our_res_cor}, Proposition \ref{our_prop_omega-stable}):\\

\noindent
{\bf Main Theorem.} Let $R$ be a Noetherian domain and $A$ an $\A{n}$-fibration over $R$. Let $B=R[W_1, W_2, \cdots, W_m]$ be a subring of $A$ such that $A \otimes_R k(P)$ is a polynomial algebra in $n-m$ variables over $B \otimes_R k(P)$ for every prime ideal $P$ of $R$. Then $A$ is an $\A{n-m}$-fibration over $B$. Moreover, the following statements are equivalent:

\begin{enumerate}
\item [\rm (1)]$\Omega_R(A)$ is stably free over $A$.
\item [\rm (2)]$A$ is a stably polynomial algebra over $B$.
\end{enumerate}

In \cite{BD_RES}, Bhatwadekar-Dutta also observed the following result (\cite{BD_RES}, Theorem 3.2):

\begin{thm} \label{thm-BD-res-iff-var}
Let $R$ be a Noetherian domain such that either $R$ contains $\mathbb{Q}$ or $R$ is seminormal. Then an element $W$ of $R[X,Y]$ is a residual variable of $R[X,Y]$ over $R$ if and only if $W$ is a variable of $R[X,Y]$.  
\end{thm}

We shall see that Theorem \ref{thm-BD-res-iff-var} holds if we replace the polynomial ring by an $\A{2}$-fibration whose module of differentials is stably free (Corollary \ref{our_res_thm}).

\section{Preliminaries} \label{sec_prelim}
Throughout the article rings will be commutative with unity. For a ring $R$, $R^{[n]}$ will denote the polynomial ring in $n$ variables over $R$. We shall use the notation $A = R^{[n]}$ to mean that $A$ is isomorphic, as an $R$-algebra, to a polynomial ring in $n$ variables over $R$.\\

For a prime ideal $P$ of $R$, $k(P)$ will denote the residue field $R_P/PR_P$. A finitely generated flat $R$-algebra $A$ is said to be an $\A{n}$-fibration over $R$ if $A \otimes_R k(P) = k(P)^{[n]}$ {\it for every prime ideal $P$ of $R$}.\\
Note that several algebraic geometers use the term ``$\A{n}$-fibration'' to mean that $A/\m A \cong (R/\m R)^{[n]}$ 
for almost all maximal ideals $\m$ of $R$. 
We emphasise that we shall use it as defined by Sathaye in \cite{Sat_Pol-two-var-DVR}, where the hypothesis is made on {\it all} fibre rings. 

\smallskip

An $R$-algebra $A$ is said to be stably polynomial algebra over an $R$-subalgebra $B$ of $A$ if there exist indeterminates $Z_1, Z_2, \cdots, Z_r$ over $A$ and indeterminates $T_1, T_2, \cdots, T_s$ over $B$ such that $A[Z_1, Z_2, \cdots, Z_r] = B[T_1, T_2, \cdots, T_s]$ (as $B$-algebras). We state below an elementary observation on stably polynomial algebras.

\begin{lem} \label{our_lem_stably-free}
Let $A$ be a stably polynomial algebra over $R$. Then $\Omega_R(A)$ is a stably free $A$-module.
\end{lem}
\begin{proof}
Set $D:= A^{[m]} = R^{[n+m]}$ for some $m,n$ and $M:= \Omega_R(A) \oplus A^m$. Then $M \otimes_A D \cong (\Omega_R(A) \otimes_A D) \oplus D^m \cong \Omega_R(D) \cong D^{n+m}$ (cf. \cite{Matsumura_Algebra}, Example 26.J, Pg. 189). Thus $M \otimes_A D$ is a free $D$-module. Since $A$ is a retract of $D = A^{[m]}$, it follows that $M$ is a free $A$-module. Thus $\Omega_R(A)$ is stably free over $A$.
\end{proof}
The following structure theorem on affine fibrations is due to T. Asanuma (\cite{Asanuma_fibre_ring}, Theorem 3.4):
\begin{thm} \label{asa_struct-fib-th}
Let $R$ be a Noetherian ring and $A$ an $\A{r}$-fibration over $R$. 
Then $\Omega_R(A)$ is a projective $A$-module of rank $r$ and $A$ is an $R$-subalgebra 
(up to an isomorphism) of a polynomial ring $R^{[m]}$ for some $m$ such that 
$A^{[m]} = \mbox{Sym}_{R^{[m]}} (\Omega_R(A) \otimes_A R^{[m]})$ (as $R$-algebras). 
In particular, if $\Omega_R(A)$ is a stably free $A$-module, then $A$ is a stably polynomial algebra over $R$.
\end{thm}

We shall use the following result by Quillen-Suslin (\cite{Quillen_proj}, \cite{Suslin_proj-free})
\begin{thm} \label{Quillen_Suslin}
If $R$ is a PID, then any finitely generated projective $R^{[n]}$-module is free.
\end{thm}

We record the following result on cancellation by Hamann (\cite{Haman_Invariance}, Theorem 2.6 and Theorem 2.8).
\begin{thm} \label{Hamann}
Let $R$ be a Noetherian ring such that either $R$ contains $\mathbb{Q}$ or $R_{red}$ is seminormal. Then $R^{[1]}$ is $R$-invariant, i.e., if an $R$-algebra $A$ is such that $A^{[m]} = R^{[m+1]}$, then $A = R^{[1]}$.
\end{thm}

The following result was first proved by Kambayashi-Miyanishi in (\cite{Kam-Miya_flat-fibration}, Theorem 1). Since any rank-one projective module over a factorial domain is free, this result can now also be seen to follow from Theorem \ref{asa_struct-fib-th} and Theorem \ref{Hamann}. A more general version of the result is given in (\cite{D_MOR}, Theorem 3.4).

\begin{thm} \label{Kam-Miya}
Let $R$ be a Noetherian factorial domain and
$A$ an $\A{1}$-fibration over R. Then $A=R^{[1]}$.
\end{thm}

\section {Main Theorem}
\begin{defn}
Let $R$ be a ring, $A$ an $R$-algebra, $n \in \mathbb{N}$ and $\ul{W}:=(W_1, W_2, \cdots, W_m )$ an $m$-tuple of elements in $A$ which are algebraically independent over $R$ such that $A \otimes_R k(P) = (R[\ul{W}] \otimes_R k(P))^{[n-m]}$ for all $P \in$ Spec($R$). We shall call such an $m$-tuple $\ul{W}$ to be an $m$-tuple residual variable of $A$ over $R$.
\end{defn}

\begin{rem}
(\cite{BD_RES}, Example 4.1) provides an example of a residual variable in a polynomial ring which is not a variable. For an example of a residual variable in an affine fibration which is not a polynomial ring,
see Example \ref{example-1} or Remark \ref{example-2}.
\end{rem} 

We first observe an elementary result.

\begin{lem}\label{our_lem_res_var}
Let $R$ be a ring, $A$ an $R$-algebra, $B_1, B_2$ $R$-subalgebras of $A$ and $B=B_1 \otimes_R B_2$. Suppose that $A \otimes_R k(P) = B \otimes_R k(P)^{[n]}$ for all $P \in $ Spec($R$). Then $A \otimes_{B_2} k(Q) = B \otimes_{B_2} k(Q)^{[n]}$ for all $Q \in $ Spec($B_2$).
\end{lem}

\begin{proof}
Fix $Q \in$ Spec($B_2$) and set $P := Q \cap R \in$ Spec($R$). Set $\ol{B}_2:=B_2 \otimes_R k(P)$. Then $B \otimes_{B_2} \ol{B}_2 = B \otimes_R k(P)$ and $A \otimes_{B_2} \ol{B}_2 = A \otimes_R k(P)$ so that

\smallskip

$A \otimes_{B_2} k(Q)$ 
   = $A \otimes_{B_2} \ol{B}_2 \otimes_{\ol{B}_2} k(Q)$
   = $A \otimes_{R} k(P) \otimes_{\ol{B}_2} k(Q)$
   = $(B \otimes_{R} k(P) \otimes_{\ol{B}_2} k(Q))^{[n]}$
   = $(B \otimes_{B_2} \ol{B}_2 \otimes_{\ol{B}_2} k(Q))^{[n]}$
   = $(B \otimes_{B_2} k(Q))^{[n]}$.
\end{proof}

As a consequence, we have
\begin{rem}
If $(\ul{U}, \ul{V}):=(U_1, U_2, \cdots, U_s, V_1, V_2, \cdots, V_t)$ is an $(s+t)$-tuple residual variable of $A$ over $R$, then $\ul{V}$ is a $t$-tuple residual variable over $R[\ul{U}]$.
\end{rem}

Next we record a result on flatness.
\begin{lem} \label{our_lemma_2}
Let $R \subset  B \subset A$ be Noetherian rings such that
\begin{enumerate}
\item [\rm (i)] $A$ and $B$ are flat over $R$.
\item [\rm (ii)]$A \otimes_R k(P)$ is flat over  $B\otimes_R k(P)$ for all $P \in$ Spec($R$).
\end{enumerate}

Then $A$ is flat over $B$.
\end{lem}

\begin{proof}
We shall show that $A_{Q}$ is flat over $B_{Q \cap B}$ for all $Q \in$ Spec($A$). Fix $Q \in$ Spec($A$) and set $P':= Q \cap B \in$ Spec($B$) and $P = P' \cap R \in$ Spec($R$). Then we have local homomorphisms $R_P \longrightarrow B_{P'} \longrightarrow A_Q$. As $B_{P'}$ is flat over $R_P$, to show that $A_Q$ is flat over $B_{P'}$, it is enough to show that $A_Q \otimes_{R_P} k(P)$ is flat over $B_{P'} \otimes_{R_P} k(P)$ (cf. \cite{Matsumura_Algebra}, 20.G, Pg. 152).
 
 \medskip
 
 Since $A \otimes_R k(P)$ is flat over $B \otimes_R k(P)$ and $Q \cap R = P$, we see that $A_Q \otimes_{R_P} k(P)$ is flat over $B \otimes_R k(P)$ and hence  $(A_Q \otimes_{R_P} k(P)) \otimes_{B} B_{P'}$ is flat over $(B \otimes_R k(P)) \otimes_{B} B_{P'}$. Now 
 
 \medskip
 
  $(A_Q \otimes_{R_P} k(P)) \otimes_{B} B_{P'}$
  = $(A_Q  \otimes_{B} B_{P'}) \otimes_{R_P} k(P)$
  = $(A_Q  \otimes_{B_{P'}} B_{P'}) \otimes_{R_P} k(P)$
  = $A_Q \otimes_{R_P} k(P)$
  
  \smallskip
  
   and
   
  \smallskip
  
  $(B \otimes_R k(P)) \otimes_{B} B_{P'}$
  = $B_{P'}\otimes_R k(P)$
  = $B_{P'}\otimes_{R_P} k(P)$.
  
  \medskip
  
  This shows that $A_Q \otimes_{R_P} k(P)$ is flat over $B_{P'} \otimes_{R_P} k(P)$ and hence $A_Q$ is a flat $B_{P'}$-algebra. Thus $A$ is a flat $B$-algebra.
\end{proof}
From Lemma \ref{our_lem_res_var} and Lemma \ref{our_lemma_2} it follows that if $\ul{W}$ is an $m$-tuple residual variable of an $\A{n}$-fibration $A$ over a ring $R$, then $A$ is an $\A{n-m}$-fibration over $R[\ul{W}]$. More generally, we have:

\begin{cor} \label{our_res_prop}
Let $R \subset B \subset A$ be Noetherian rings such that $A$ is an $\A{n}$-fibration over $R$ and
$B$ an $\A{m}$-fibration over $R$ with $A \otimes_R k(P) = {B \otimes_R k(P)}^{[n-m]} $ for all $P \in$ Spec($R$). Then $A$ is an $\A{n-m}$-fibration over $B$.
\end{cor}

\begin{rem} \label{our_res_iff_A2-fib}
Let $A$ be an $\A{3}$-fibration over $R$ and $W \in A$. It was shown in (\cite{Kahoui_Residual}, Proposition 3.2) that if $\mathbb{Q} \hookrightarrow R$ and $A$ an $\A{2}$-fibration over $R[W]$, then $W$ is a residual variable of $A$ over $R$. The converse was also proved for the case $R$ is a regular affine domain over $\mathbb{C}$ (\cite{Kahoui_Residual}, Theorem 3.3). Corollary \ref{our_res_prop} shows that the converse holds for {\it any} Noetherian domain.
\end{rem}

As a consequence of Corollary \ref{our_res_prop} and Theorem \ref{Kam-Miya}, we see that in an $\A{m+1}$-fibration over a Noetherian factorial
domain, any $m$-tuple residual variable is necessarily a variable.
\begin{cor} \label{our_rem1-res}
Let $R$ be a Noetherian factorial domain and $A$ an $\A{m+1}$-fibration over $R$. Then an $m$-tuple $\ul{W}$ of $A$ is an $m$-tuple residual variable of $A$ over $R$ if and only if  $A = R[\ul{W}]^{[1]} = R^{[m+1]}$.
\end{cor}

\begin{rem}
Note that Corollary \ref{our_rem1-res} need not hold for an $m$-tuple residual variable $\ul{W}$ of an $\A{n}$-fibration $A$ over $R$ when $n-m>1$. Example \ref{example-1} shows that $A$ may not be even a stably polynomial algebra over $R[\ul{W}]$. The next result shows that if $R$ is a polynomial algebra over a PID, then $A$ happens to be a stably polynomial algebra over $R[\ul{W}]$.
\end{rem}

\begin{cor} \label{our_thm_pid}
Let $R$ be a finite-dimensional polynomial algebra over a PID, $A$ an $\A{n}$-fibration over $R$ and $\ul{W}$ an $m$-tuple residual variable of $A$ over $R$. Then $A$ is a stably polynomial algebra over $R[\ul{W}]$.
\end{cor}
\begin{proof}
 By Corollary \ref{our_res_prop}, $A$ is an $\A{n-m}$-fibration over $R[\ul{W}]$ and hence by (\cite{Asanuma_fibre_ring}, Corollary 3.5), $A^{[r]} = \mbox{Sym}_{R[\ul{W}]}(M)$ for some $r \in \mathbb{N}$ where $M$ is a finitely generated projective $R[\ul{W}]$ module of rank $n-m+r$. Since $R$ is a polynomial algebra over a PID, by Theorem \ref{Quillen_Suslin}, we get that $A^{[r]} = R[\ul{W}]^{[n-m+r]}$, i.e., $A$ is a stably polynomial algebra over $R[\ul{W}]$.
\end{proof}
\begin{rem}
\item [\rm (1)] Corollary \ref{our_rem1-res} shows that if $n-m=1$, then we have $A = R[\ul{W}]^{[1]}$ in Corollary \ref{our_thm_pid}. However, if $n-m>1$, then an example of Asanuma (\cite{Asanuma_fibre_ring}, Theorem 5.1) shows that we need not have $A = R[\ul{W}]^{[n-m]}$ even in the case $R$ is a PID, $m=1$ and $n=3$. In Asanuma's example, $\mathbb{Q}\nhkrghtarrw R$. When $\mathbb{Q} \hookrightarrow R$, it is not known whether, in Corollary \ref{our_thm_pid}, one can conclude that $A = R[\ul{W}]^{[n-m]}$ even in the case $m=1$ and $n=3$. For instance in (\cite{BD_AFNFIB}, Example 4.13), $W$ is a residual variable in $A = R[X,Y,Z]$, where $R$ is a DVR containing $\mathbb{Q}$, and it is not known whether $A = R[W]^{[2]}$.

\item [\rm (2)] The proof of Corollary 3.10 shows that the hypothesis that $R$ is ``a finite-dimensional polynomial algebra over a PID'' can be replaced by the condition that $R$ is ``a regular ring with trivial Grothendieck group''.
\end{rem}

The following observation on module of differentials of an $\A{n}$-fibration having residual variables is crucial for our main theorem.

\begin{lem} \label{our_lem-Omega}
Let $R$ be a Noetherian ring and $A$ an $\A{m+k}$-fibration over $R$. If $\ul{W}$ is an $m$-tuple residual 
variable of $A$ over $R$, then $\Omega_R(A) = \Omega_{R[\ul{W}]}(A) \oplus A^m$. 
In particular, $\Omega_R(A)$ is a stably free $A$-module if and only if $\Omega_{R[\ul{W}]}(A)$ is a stably free $A$-module.
\end{lem}

\begin{proof}
By Corollary \ref{our_res_prop}, $A$ is an $\A{k}$-fibration over $R[\ul{W}]$ and hence by Theorem \ref{asa_struct-fib-th}, $A$ is an $R[\ul{W}]$-subalgebra of a polynomial algebra $B$ over $R[\ul{W}]$ and $\Omega_{R[\ul{W}]}(A)$ is a projective $A$-module of rank $k$. Since $R \hookrightarrow R[\ul{W}] \hookrightarrow A \hookrightarrow B$, and since for any $A$-module $M$, every $R$-derivation $d: R[\ul{W}] \longrightarrow M$, can be extended to an $R$-derivation $\tilde{d}|_A: A \longrightarrow M$ where $\tilde{d}|_A$ is the restriction of an extension $\tilde{d}: B \longrightarrow M$ of $d$, we have the following split short exact sequence (\cite{Matsumura_Algebra}, Theorem 57, p186):

$$0 \longrightarrow A \otimes_{R[\ul{W}]} \Omega_R(R[\ul{W}]) \longrightarrow \Omega_R(A) \longrightarrow \Omega_{R[\ul{W}]}(A) \longrightarrow 0$$

This shows that $\Omega_R(A) = \Omega_{R[\ul{W}]}(A) \oplus A \otimes_{R[\ul{W}]} \Omega_R(R[\ul{W}]) = \Omega_{R[\ul{W}]}(A) \oplus A^m$.
\end{proof}

We now prove our main result.

\begin{thm} \label{our_res_cor}
Let $R$ be a Noetherian ring and $A$ an $\A{n}$-fibration over $R$ such that $\Omega_R(A)$ is a stably free $A$-module. Suppose $\ul{W}$ is an $m$-tuple residual variable of $A$ over $R$. Then $A$ is a stably polynomial algebra over $R[\ul{W}]$; specifically, $A^{[\ell]} = R[\ul{W}]^{[n-m+\ell]}$ for some $\ell \in \mathbb{N}$.
\end{thm}
\begin{proof}
By Corollary \ref{our_res_prop}, $A$ is an $\A{n-m}$-fibration over $R[\ul{W}]$ and hence, by Lemma \ref{our_lem-Omega}, $\Omega_{R[\ul{W}]}(A)$ is a stably free $A$-module. Therefore, we get the result by Theorem \ref{asa_struct-fib-th}.
\end{proof}

The following example shows the necessity of the assumption ``$\Omega_R(A)$ is a stably free $A$-module'' in Theorem \ref{our_res_cor} even when R is a regular factorial affine domain over the field of
real numbers.
\begin{ex} \label{example-1}
Let $R={\mathbb R}[X,Y,Z]/(X^2+Y^2+Z^2-1)$. It is
well-known that $K_0(R)$, the Grothendieck group of $R$, is non-trivial (in fact, it is
${\mathbb Z} \oplus {\mathbb Z}/(2)$); in particular, there exists a finitely generated projective 
$R$-module $M$ of rank $2$ which is not stably free. Let $W$ be an indeterminate over $R$ and 
$A=\mbox{Sym}_R (M \oplus RW)$. Then $A_P=R_P[W]^{[2]}$ for all $P \in$ Spec($R$) so that $A$ is an $\A{3}$-fibration over $R$ with $W$ as a residual variable.  If $A^{[\ell]} =R[W]^{[\ell +2]}$, then we would have $\mbox{Sym}_R (M \oplus RW \oplus R^{\ell}) \cong \mbox{Sym}_R (RW \oplus R^{\ell +2})$ and hence, by (\cite{Eakin-Heinzer_A-Cncl-prob}, Lemma 1.3), we would have $M \oplus RW \oplus R^{\ell} \cong RW \oplus R^{\ell +2}$ contradicting that $M$ is not stably free.
\end{ex} 

\begin{rem} \label{example-2}
When $R$ is a Noetherian factorial domain and $m=n-1$
then Corollary \ref{our_rem1-res} shows that the hypothesis ``$\Omega_R(A)$ is stably
free'' may be dropped from Theorem \ref{our_res_cor}; in fact, in this case, $A=R[\ul{W}]^{[1]}$. But even over a (non-factorial) Dedekind domain and even for $n=2$ and $m=1$, an $\A{2}$-fibration need not be stably polynomial over $R[W]$ when $W$ is a residual variable of $A$. For instance, choose a non-principal ideal $I$ of the Dedekind domain $R$ and set $A=\mbox{Sym}_R (I \oplus RW)$. As in Example \ref{example-1}, $A$ is an $\A{2}$-fibration over $R$, $W$ is a residual variable of $A$ but $A$ is not stably polynomial over $R[W]$.
\end{rem}

The following result gives a converse of Theorem \ref{our_res_cor}.

\begin{prop} \label{our_prop_omega-stable}
Let $R$ be a Noetherian ring and $A$ an $\A{n}$-fibration over $R$. Suppose that there exists an $m$-tuple residual variable $\ul{W}$ of $A$ over $R$ such that $A$ is a stably polynomial algebra over $R[\ul{W}]$. Then $\Omega_R(A)$ is a stably free $A$-module.
\end{prop}

\begin{proof}
By Lemma \ref{our_lem-Omega}, it sufficies to show that $\Omega_{R[\ul{W}]}(A)$ is a stably free $A$-module. This follows from Lemma \ref{our_lem_stably-free}.
\end{proof}
For convenience, we state below an easy result.

\begin{lem} \label{our_prop_fib_imply_res}
Let $R \hookrightarrow B \hookrightarrow A$ be integral domains such that $A$ is an $\A{n}$-fibration over $B$. Then $A \otimes_R k(P)$ is an $\A{n}$-fibration over $B \otimes_R k(P)$ for every $P \in$ Spec($R$) . Moreover, if $B=R^{[m]}$, then $A \otimes_R k(P)$ is a stably polynomial algebra over $B \otimes_R k(P)$ for each $P \in$ Spec($R$).
\end{lem}
\begin{proof}
Let $P \in$ Spec($R$), $\ol{A} = A \otimes_R k(P)$ and $\ol{B} = B \otimes_R k(P)$. Since $A$ is a finitely generated flat $B$-algebra, clearly $\ol{A}$ is a finitely generated flat $\ol{B}$-algebra. Note that $\ol{A} = A \otimes_R k(P) = A \otimes_B B \otimes_{R} k(P) = A \otimes_B \ol{B}$. Let $\ol{Q} \in$ Spec($\ol{B}$) and let $Q$ be the prime ideal of $B$ for which $Q \ol{B} = \ol{Q}$. Then $Q \cap R = P$. Now $k(Q)$ is the field of fractions of $B/Q$ and hence of $B_P/QB_P \cong \ol{B}/\ol{Q}$; thus $k(Q)= k(\ol{Q})$. Hence, $\ol{A} \otimes_{\ol{B}} k(\ol{Q}) = A \otimes_B \ol{B} \otimes_{\ol{B}} k(\ol{Q}) = A \otimes_B k(Q) = k(Q)^{[n]} = k(\ol{Q})^{[n]}$.
\end{proof}

\begin{cor} \label{our_res_iff_A1-fib}
Let $R$ be a Noetherian domain and $A$ an $\A{m+1}$-fibration over $R$. Then an $m$-tuple $\ul{W}$ of $A$ is an $m$-tuple residual variable of $A$ over $R$ if and only if $A$ is an $\A{1}$-fibration over $R[\ul{W}]$.
\end{cor}

\begin{proof}
Follows from Corollary \ref{our_res_prop}, Lemma \ref{our_prop_fib_imply_res} and Corollary \ref{our_rem1-res}.
\end{proof}

Finally, we give a condition for an $m$-tuple residual variable in an $\A{m+1}$-fibration to be a variable.
\begin{cor} \label{our_res_thm}
Let $R$ be a Noetherian domain and $A$ an $\A{m+1}$-fibration over $R$ such that $\Omega_R(A)$ is a stably free $A$-module. Suppose that either $R$ contains $\mathbb{Q}$ or $R$ is seminormal. Then for an $m$-tuple $\ul{W}$ of $A$, the following are equivalent:
\begin{enumerate}
\item [\rm (I)] $\ul{W}$ is an $m$-tuple residual variable of $A$ over $R$.
\item [\rm (II)] $A = R[\ul{W}]^{[1]} = R^{[m+1]}$.
\end{enumerate}
\end{cor}

\begin{proof} Follows from Theorem \ref{our_res_cor} and Theorem \ref{Hamann}.
\end{proof}

\section{Appendix}
Theorem \ref{our_res_cor} can be slightly generalised when envisaged as a statement on ``tensor product decomposition'' discussed in \cite{Bass_Wright}. 
Let $R\subset B\subset A$ be Noetherian rings with $A$ an $\A{n}$-fibration over $R$
and also an $\A{n-m}$-fibration over $B$. Theorem \ref{our_res_cor} shows that if 
$B= R^{[m]}$ and $\Omega_R(A) $ is stably free then $A^{[\ell]}= B^{[\ell]} \otimes_{R^{[\ell]}} C$, where $C= R^{[n-m+\ell]}$. 
One can see below (Proposition \ref{p1}) that even when $B$ is only an $\A{m}$-fibration over $R$, it is stably
a factor in a tensor product decomposition of $A$, even without the hypothesis that $\Omega_R(A) $ is stably free. 
We first note that the proof of Lemma \ref{our_lem-Omega} can be seen to yield the following general version:

\begin{lem}\label{l1}
Let $R \subset B \subset A$ be Noetherian rings such that $A$ is an $\A{n}$-fibration over $R$ and
$B$ an $\A{m}$-fibration over $R$ with $A \otimes_R k(P) = {B \otimes_R k(P)}^{[n-m]} $ for all $P \in$ Spec($R$).
Then $\Omega_R(A) = \Omega_{B}(A) \oplus (\Omega_R(B)\otimes_{B} A)$. 
\end{lem}

As a consequence we have 
\begin{prop}\label{p1}
Under the hypothesis of Lemma \ref{l1}, there exists $\ell \ge 0$ such that 
$A^{[\ell]}= C \otimes_{R^{[\ell]}}{B^{[\ell]}}$ for some $R^{[\ell]}$-algebra $C$. 
\end{prop}
\begin{proof}
 Since $A$ and $B$ are affine fibrations over $R$, by Theorem \ref{asa_struct-fib-th}, we can choose a sufficiently large positive integer $\ell$
such that 
$$
A^{[\ell]} = \mbox{Sym}_{R^{[\ell]}} (\Omega_R(A) \otimes_A R^{[\ell]}) \text{~~and~~}
B^{[\ell]} = \mbox{Sym}_{R^{[\ell]}} (\Omega_R(B) \otimes_B R^{[\ell]}).
$$
Therefore,  by Lemma \ref{l1}, we have 
\begin{eqnarray*}
 A^{[\ell]} &=& \mbox{Sym}_{R^{[\ell]}} (\Omega_R(A) \otimes_A R^{[\ell]})\\
&=& \mbox{Sym}_{R^{[\ell]}} ((\Omega_{B}(A) \oplus (\Omega_R(B)\otimes_{B} A)) \otimes_A R^{[\ell]})\\
&=& \mbox{Sym}_{R^{[\ell]}} (\Omega_{B}(A) \otimes_A R^{[\ell]})\otimes_{R^{[\ell]}} 
\mbox{Sym}_{R^{[\ell]}} ((\Omega_R(B)\otimes_{B} A)\otimes_A R^{[\ell]})\\
&=& \mbox{Sym}_{R^{[\ell]}} (\Omega_{B}(A) \otimes_A R^{[\ell]})\otimes_{R^{[\ell]}} 
\mbox{Sym}_{R^{[\ell]}} (\Omega_R(B)\otimes_{B} R^{[\ell]})\\
&=& \mbox{Sym}_{R^{[\ell]}} (\Omega_{B}(A) \otimes_A R^{[\ell]})\otimes_{R^{[\ell]}} B^{[\ell]}\\
&=& C \otimes_{R^{[\ell]}} B^{[\ell]},
\end{eqnarray*}
where $C= \mbox{Sym}_{R^{[\ell]}} (\Omega_{B}(A) \otimes_A R^{[\ell]})$. 
\end{proof}

\smallskip

\noindent
{\bf Acknowledgements:} The authors thank Neena Gupta for carefully going through the draft and the referee for suggestions which have been incorporated in the Appendix.

\bibliography{reference}
\bibliographystyle{amsalpha} 
\end{document}